\documentclass[12pt]{amsart}

\usepackage{latexsym,enumerate,graphicx}
\usepackage{amsmath,amsthm,amsopn,amstext,amscd,amsfonts,amssymb}
\usepackage{fullpage,yhmath}
\usepackage{hyperref}
\usepackage{eucal}
\usepackage{color}
\usepackage{verbatim}
\usepackage[normalem]{ulem}

\numberwithin{equation}{section}

\newtheorem{theorem}{Theorem}
\newtheorem{lemma}{Lemma}
\newtheorem{corollary}{Corollary}
\newtheorem{proposition}{Proposition}
\newtheorem{remark}{Remark}

\numberwithin{theorem}{section}
\numberwithin{corollary}{section}
\numberwithin{lemma}{section}
\numberwithin{definition}{section}
\numberwithin{proposition}{section}
\numberwithin{remark}{section}

\newcommand{\R}{\mathbb R}
\newcommand{\N}{\mathbb N}

\newcommand{\medint}{-\kern  -,375cm\int}
\newcommand{\dint}{\displaystyle\int}

\makeatletter
\@namedef{subjclassname@2020}{%
  \textup{2020} Mathematics Subject Classification}
\makeatother

\begin{document}

\title{Eigenvalue estimates for $p$-Laplace problems on domains expressed in Fermi Coordinates}
\author{B. Brandolini$^{\ast}$ - F. Chiacchio$^\dagger$  - J. J. Langford$^\ddagger$}

\date{}

\maketitle
\footnotesize{
\centerline{$^\ast$ Dipartimento di Matematica e Informatica}
\centerline{Universit\`a degli Studi di Palermo}
\centerline{via Archirafi 34, 90123 Palermo, Italy}
\centerline{\texttt{ barbara.brandolini@unipa.it}}

\bigskip

\centerline{$^\dagger$ Dipartimento di Matematica e Applicazioni ``R. Caccioppoli''}
\centerline{Universit\`a degli Studi di Napoli Federico II}
\centerline{Monte S. Angelo, via Cintia, I-80126 Napoli, Italy}
\centerline{\texttt{francesco.chiacchio@unina.it}}

\bigskip
\centerline{$^\ddagger$ Department of Mathematics, Bucknell University}
\centerline{One Dent Drive, Lewisburg, PA 17837, USA}
\centerline{\texttt{jeffrey.langford@bucknell.edu}
}

\begin{abstract}
We prove explicit and sharp eigenvalue estimates for Neumann $p$-Laplace eigenvalues in domains that admit a representation in Fermi coordinates. More precisely, if $\gamma$ denotes a non-closed curve in $\mathbb{R}^2$ symmetric with respect to the $y$-axis, let $D\subset \mathbb{R}^2$ denote the domain of points that lie on one side of $\gamma$ and within a prescribed distance $\delta(s)$ from $\gamma(s)$ (here $s$ denotes the arc length parameter for $\gamma$). Write $\mu_1^{odd}(D)$ for the lowest nonzero eigenvalue of the Neumann $p$-Laplacian with an eigenfunction that is odd with respect to the $y$-axis. For all $p>1$, we provide a lower bound on $\mu_1^{odd}(D)$ when the distance function $\delta$ and the signed curvature $k$ of $\gamma$ satisfy certain geometric constraints. In the linear case ($p=2$), we establish sufficient conditions to guarantee $\mu_1^{odd}(D)=\mu_1(D)$. We finally study the asymptotics of $\mu_1(D)$ as the distance function tends to zero. We show that in the limit, the eigenvalues converge to the lowest nonzero eigenvalue of a weighted one-dimensional Neumann $p$-Laplace problem.
\end{abstract}

\section{Introduction and Main Results}

Suppose that $D\subset \R^2$ is a bounded, Lipschitz domain and $p>1$. We consider the Neumann eigenvalue problem for the $p$-Laplace operator, that is,
\begin{equation}\label{NP}
\left\{
\begin{array}{ll}
-\Delta_p u=\mu |u|^{p-2}u &\mbox{in}\>D,
\\ \\
|\nabla u|^{p-2}\frac{\partial u}{\partial \mathbf{n}}=0 &\mbox{on}\> \partial D,
\end{array}
\right.
\end{equation}
where $\Delta_p=\textup{div}(|\nabla u|^{p-2}\nabla u)$ denotes the $p$-Laplace operator and $\frac{\partial u}{\partial \mathbf{n}}$ denotes the outer normal derivative. We are interested in estimating the first nonzero eigenvalue $\mu_1(D)$ of \eqref{NP}. As is well-known, this eigenvalue is characterized variationally as
\begin{equation}\label{varchar}
\mu_1(D)=\min\left\{\dfrac{\dint_D |\nabla v|^p\,dxdy}{\dint_D |v|^p\,dxdy} :\> v \in W^{1,p}(D)\setminus\{0\}, \> \int_D |v|^{p-2}v\, dxdy=0 \right\};
\end{equation}
for more on the $p$-Laplace operator we refer the interested reader to \cite{L} and the references therein.

In the present paper, we focus our attention on a specific class of domains $D$ that have a line of symmetry, chosen to be the $y$-axis. Our  domains are non-convex in general. To set the scene, say $\gamma(s)=(x(s),y(s)),\> s \in [0, L],$ is a smooth, simple, non-closed curve, parametrized by its arc length, and symmetric with respect to the $y$-axis. The coordinate functions of $\gamma$ therefore satisfy
$$
x(L-s)=-x(s), \quad y(L-s)=y(s), \quad s \in \left[0,\frac{L}{2}\right].
$$
We denote by $k(s)$ the signed curvature of $\gamma(s)$. We consider domains constructed via Fermi (or parallel) coordinates, using $\gamma$ as a frame of reference (see also \cite{EHL,EL}). If $\delta:[0,L]\to (0,\infty)$ is smooth and even with respect to $s=\frac{L}{2}$, we may describe our domains $D$ as follows:  
\begin{equation} \label{D}
D=
\left\{(x(s)+ry'(s), y(s)-rx'(s)): \> s \in (0,L), \> r \in (0,\delta(s))\right\}.
\end{equation}
Geometrically, the domain $D$ is constructed as follows. We find a normal vector to $\gamma(s)$ by rotating the tangent vector $\gamma'(s)$ by $\frac{\pi}{2}$ radians clockwise. For each point along the curve $\gamma(s)$, one moves along this normal direction a positive distance less than $\delta(s)$. See Figure 1.

\begin{figure}[h]
\includegraphics[scale=0.3]{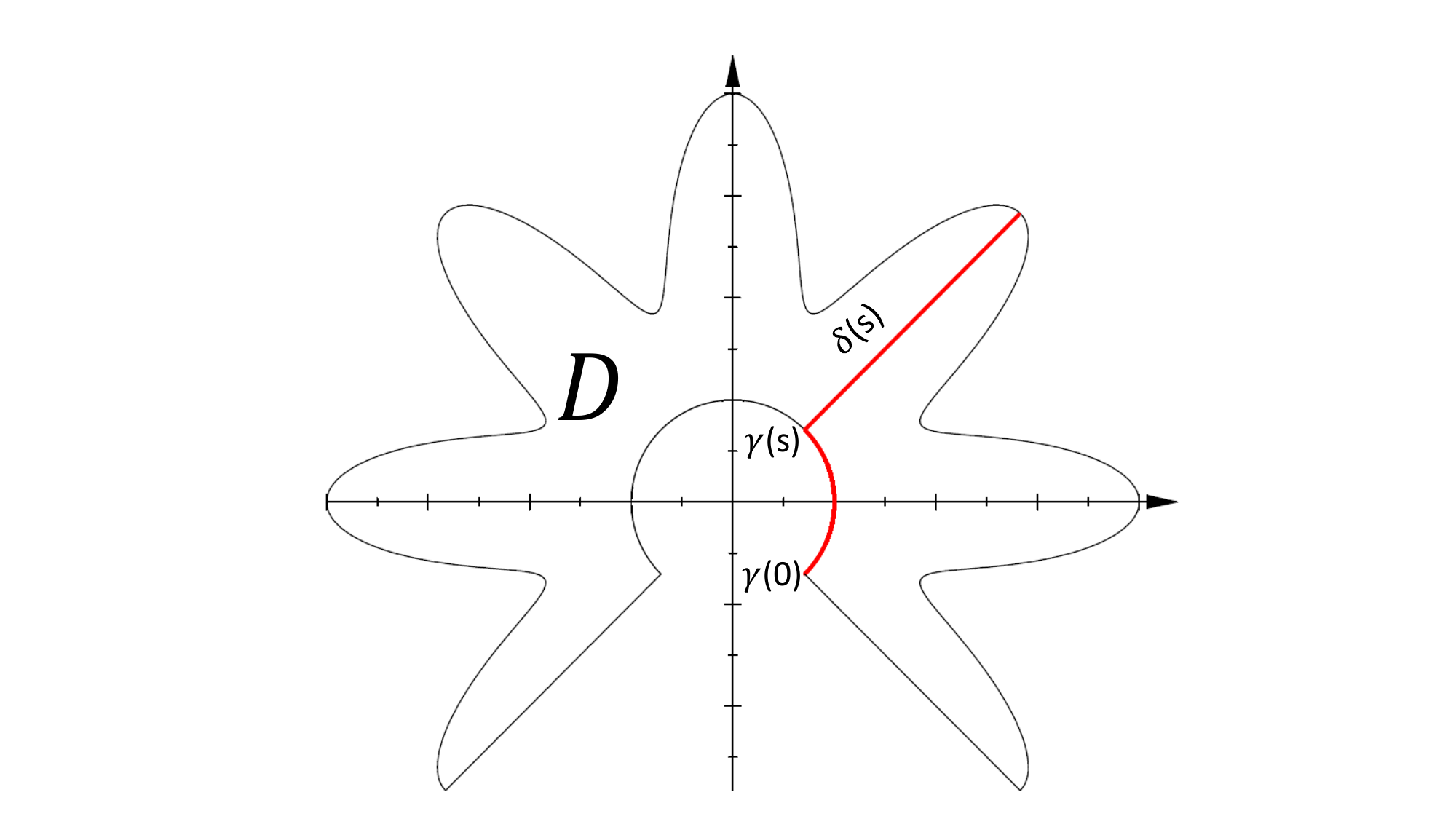}
\caption{}
\end{figure}

In  \cite{BCDL} we estimated $\mu_1(D)$ from below for domains as in \eqref{D} when $p=2$ and  $\delta$ is constant. The aim of the present paper is to generalize the results contained in \cite{BCDL} along several directions. Firstly,  the set $D$ may  have nonconstant width, i.e. here $\delta=\delta(s)$ is a function of the arc length parameter. Secondly,  we obtain results for the nonlinear $p$-Laplace operator. Finally, we are also interested in discovering the asymptotic behavior of $\mu_1(D)$ as $\delta$ goes to 0.

 In order to state properly our main results, we denote by 
$\mu_1^{odd}(D)$ the first nonzero eigenvalue of \eqref{NP}, having a corresponding eigenfunction which is odd with respect to the $y$-axis.  The existence of such an eigenfunction and its variational characterization are given, for instance, in \cite{CP,KL}.
Let $D^F$ be the domain $D$ written in Fermi coordinates, that is,
$$
D^F=
\{(s,r)\in \R^2: \> 0 <s < L, \>0 < r < \delta(s)\}.
$$
For the entirety of our paper, we assume that the Fermi transformation is one-to-one from $\overline{D^F}$ to $\overline{D}$.

Our  first main result establishes a sufficient condition for $\mu_1(D)=\mu_1^{odd}(D)$ in the linear case $p=2$.

\begin{theorem}\label{t2}
 Let $p=2$ and assume $1+rk(s)>0$ on $\overline {D^F}$.   
Then 
$$\mu_1(D)=\mu _{1}^{odd}(D)$$
if any of the following conditions are satisfied:
\begin{enumerate}
 \item[(a)]  $k(s)\ge0$ for all $s \in [0,L]$ and 
 \begin{equation}\label{delta}
\mu_1(D)<\frac{1}{\underset{s\in[0,L]}{\max} (2\delta(s)+\delta(s)^2 k(s))^2}; 
\end{equation}

\item[(b)] $k(s)<0$ for all $s \in [0,L]$ and
\begin{equation}\label{delta1}
\mu_1(D)<\dfrac{1}{\underset{s\in[0,L]}{\max}\dfrac{4\delta(s)^2}{ (1+\delta(s) k(s))^2}};
\end{equation}

\item[(c)] $k(s)$ changes its sign in $[0,L]$, and
\begin{equation}\label{delta2}
\mu_1(D)<\dfrac{1}{\max\left\{  \underset{s\in[0,L]}{\max} (2\delta(s)+\delta(s)^2 k(s))^2,  \underset{s\in[0,L]}{\max}\dfrac{4\delta(s)^2}{ (1+\delta(s) k(s))^2}\right\}}.
\end{equation}

\end{enumerate}
\end{theorem}

We next establish a lower bound on $\mu_1^{odd}(D)$ so long as the distance and curvature functions satisfy certain concavity assumptions. Our lower bound depends on whether the distance function $\delta(s)$ is constant or not.

\begin{theorem}[Constant $\delta$]\label{t11}
Let $p>1$ and assume $\delta(s)=\delta$ is constant. If $k(s)$ is concave on $[0,L]$ with $1+rk(s)>0$ on $\overline{D^F}$, then 
\begin{equation}\label{est11}
\mu_1^{odd}(D) \ge A_p \left(\frac{\pi_p}{L}\right)^p,
\end{equation}
where
\begin{equation}\label{A_p}
A_p=\displaystyle \min_{ \overline{D^F}}\frac{1}{(1+rk(s))^p}
\end{equation}
and
\begin{equation}\label{pi_p}
\pi_p=2\dint_0^{+\infty} \frac{1}{1+\frac{1}{p-1}s^p}\,ds=2\pi \frac{(p-1)^{1/p}}{p\sin (\pi/p)}.
\end{equation}
\end{theorem}

\begin{theorem}[Nonconstant $\delta$]\label{t1}
Let $p>1$ and assume $\delta(s)$ is nonconstant. Suppose that $\delta(s)$ and either $\delta(s)\,k(s)$ or $\delta(s)^2k(s)$ are concave functions on $[0,L]$, and that $|\delta'(s)|\le 1$  in $[0,L]$. If $1+rk(s)>0$ on $\overline{D^F}$, then 
\begin{equation*}\label{est}
\mu_1^{odd}(D) \ge B_p \left(\frac{\pi_p}{L}\right)^p,
\end{equation*}
where 
\begin{equation}\label{Bp}
B_p= \left\{\begin{array}{ll}
2^{-p/2}\min\left\{1,\displaystyle{\min_{ \overline{D^F}}}\frac{1}{(1+rk(s))^p}\right\} \quad & \mbox{if } 1<p<2,
\\
2^{1-p}\min\left\{1,\displaystyle{\min_{ \overline{D^F}}}\frac{1}{(1+rk(s))^p} \right\}\quad & \mbox{if } p \ge 2,
\end{array}
\right.
\end{equation} 
and $\pi_p$ is defined in \eqref{pi_p}.
\end{theorem}

Taken in sum, our main results yield the following corollaries.

\begin{corollary}[Constant $\delta$]\label{t33}
Let $p=2$ and assume $\delta(s)=\delta$ is constant. If $k(s)$ is concave on $[0,L]$ with $1+rk(s)>0$ on $\overline{D^F}$, and one of the conditions (a), (b), or (c) in Theorem \ref{t2} is fulfilled, then
\begin{equation}\label{bbbb}
\mu_1(D)=\mu_1^{odd}(D) \ge A_2 \left(\frac{\pi}{L}\right)^2,
\end{equation}
where 
$A_2= \displaystyle{\min_{ \overline{D^F}}}\frac{1}{(1+rk(s))^2}$. The estimate is sharp, since equality holds in \eqref{bbbb} when $D$ is a rectangle, i.e. $\gamma$ is a segment and $\delta$ is constant.
\end{corollary}

\begin{corollary}[Nonconstant $\delta$]\label{t3}
Let $p=2$ and assume $\delta(s)$ is nonconstant. Suppose that $\delta(s)$ and either $\delta(s)k(s)$ or $\delta(s)^2\,k(s)$ are concave functions on $[0,L]$, and  that $|\delta'(s)|\le 1$  in $[0,L]$. If  $1+rk(s)>0$ on $\overline{D^F}$ and one of the alternatives (a), (b) or (c) in Theorem \ref{t2} is fulfilled, then
\begin{equation*}\label{bbb}
\mu_1(D)=\mu_1^{odd}(D) \ge B_2 \left(\frac{\pi}{L}\right)^2,
\end{equation*}
where $B_2= \displaystyle \frac{1}{2} \min \left \{1,\min_{\overline{D^F}}\frac{1}{(1+rk(s))^2}\right \}$. 
\end{corollary}

 We finally establish the following result on the asymptotic behavior of $\mu_1(D)$ as the width tends to $0$. Our result requires no concavity assumptions on either $\delta(s)$ or $k(s)$.

\begin{theorem}\label{th:4}
Let $p>1$. For $\varepsilon>0$ let $D_{\varepsilon}$ denote the domain with reference curve $\gamma$ and width function $\varepsilon \delta(s)$. If $1+rk(s)>0$ on $\overline{D^F}$, then we have
$$\lim_{\varepsilon \to 0^+}\mu_1(D_{\varepsilon})=\mu_1(0,L;\delta),$$
 where $\mu_1(0,L;\delta)$ denotes the first nonzero eigenvalue of the weighted Neumann p-Laplace problem
\begin{equation}\label{eq:SLJ}
\left\{
\begin{array}{ll}
-\left(\delta |u'|^{p-2} u'\right)'=\mu \delta |u|^{p-2} u \quad &\mbox{in } (0,L),
\\ \\
u'(0)=u'(L)=0.
\end{array}
\right.
\end{equation}
\end{theorem}

We next place our work in the existing literature. By its variational characterization \eqref{varchar}, it is straightforward to obtain rough upper bounds for $\mu_1(D)$ by choosing a suitable test function. If one looks for isoperimetric upper bounds, the most celebrated result is the Szeg\"o-Weinberger inequality (see \cite{S,W} and \cite{H}). On the other hand, finding lower bounds can be much harder because, among other things, there is no monotonicity of Neumann eigenvalues with respect to set inclusion. In the linear case $p=2$ we recall the celebrated estimate by Payne and Weinberger contained in \cite{PW}. There, the authors prove
$$
\mu_1(D)\ge \left(\dfrac{\pi}{d}\right)^2,
$$
where $D$ is a convex domain in $\R^n, \> n \ge 2,$ and $d$ stands for its diameter. It is well-known that such an estimate is sharp and that the convexity assumption cannot be removed. Such a result has been generalized to the nonlinear setting in \cite{FNT}, \cite{V}, where again, the domain must be convex. It is natural to seek lower bounds on $p$-Laplace eigenvalues when $D$ is not convex. It is reasonable to expect that such an estimate will involve geometric quantities related to $D$ other
than the diameter. In \cite{BCT1,BCT2} the authors prove lower bounds for $\mu_1(D)$ in the linear and nonlinear case which involve the isoperimetric constant relative to $D$. There is also a rich line of research in this direction developed by Goldenstein, Ukhlov, and coauthors (see \cite{GU,GPU,Pc} and the references therein). They are able to provide lower bounds for simply connected planar domains in terms of Lebesgue norms of Riemann conformal mappings.  Similar estimates hold in terms of the Cheeger constant (see for instance \cite{AM,Ch,Li}). More references about eigenvalues of elliptic operators with various boundary conditions can be found, for example, in the monographs \cite{H,H1}. 

The remainder of this note is organized as follows. Section 2 is devoted to notation and preliminaries about one-dimensional weighted eigenvalue problems. The proof of Theorem \ref{t2} is contained in Section 3. In Section 4 we prove Theorem \ref{t1} using Fermi coordinates, adapting the slicing technique introduced by Payne and Weinberger in \cite{PW}. This slicing technique allows us to reduce the dimension and estimate $\mu_1^{odd}(\Omega)$ from below in terms of the first nonzero eigenvalue of a one-dimensional weighted $p$-Laplace problem.  The asymptotics of $\mu_1(D)$ as the distance function $\delta$ tends to zero are studied in Section 5.

\section{Notation and preliminary results}

We start this section with a one-dimensional lemma due to Payne and Weinberger \cite{PW}. 

\begin{lemma}\label{pw}
Let $w(s)$ be a positive, concave function on the interval $(0,L)$. Then for any piecewise  continuously differentiable function $v(s)$ that satisfies
$$
\int_0^L v(s)w(s)ds=0,
$$
it follows that
$$
\int_0^L v'(s)^2 w(s)\,ds \ge \frac{\pi^2}{L^2}\int_0^L v(s)^2 w(s)\,ds.
$$
Equivalently,
$$
\inf_{\begin{matrix} v \in W^{1,2}(0,L)\setminus\{0\} \\ \dint_0^L v(s)w(s)\,ds =0\end{matrix}} \frac{\dint_0^L v'(s)^2 w(s)\,ds}{\dint_0^L v(s)^2 w(s)\,ds} \ge \min_{\begin{matrix} v \in W^{1,2}(0,L)\setminus\{0\} \\ \dint_0^L v(s)\,ds=0\end{matrix}} \frac{\dint_0^L v'(s)^2\,ds}{\dint_0^L v(s)^2\,ds}=\frac{\pi^2}{L^2}.
$$
\end{lemma}

An extension of the previous lemma to the nonlinear framework is contained in \cite{FNT} and can be stated as follows (recall the definition of $\mu_1(0,L;w)$ in \eqref{eq:SLJ}).

\begin{lemma}\label{2}
Let $w(s)$ be a  positive, log-concave function on the interval $(0,L)$, and let $p>1$. Then 
\begin{eqnarray}\label{fnt}
 \mu_1(0,L;w)&=&\inf_{\begin{matrix} v \in W^{1,p}(0,L)\setminus\{0\} \\ \dint_0^L |v|^{p-2}v\,w\,ds =0\end{matrix}} \frac{\dint_0^L |v'(s)|^p w(s)\,ds}{\dint_0^L |v(s)|^p w(s)\,ds} 
 \\
 &\ge& \min_{\begin{matrix} v \in W^{1,p}(0,L)\setminus\{0\} \\ \dint_0^L |v|^{p-2}v\,ds=0\end{matrix}} \frac{\dint_0^L |v'(s)|^p\,ds}{\dint_0^L |v(s)|^p\,ds}=\left(\frac{\pi_p}{L}\right)^p, \notag
\end{eqnarray}
with $\pi_p$ defined as in \eqref{pi_p}.
\end{lemma}

\smallskip

We next provide a Lyapunov-type estimate for $\mu_1(0,L;w)$ that complements Lemma \ref{2}. On one hand, the Lyapunov estimate is a bit weaker than \eqref{fnt}, but on the other, it does not require a concavity assumption on the weight. 
Results of this form are well-known, though we cannot find a precise statement of this particular result in the literature. For convenience, we include its short proof, since several of the ideas presented here are used in the proof of Theorem \ref{th:4}. For related results see, for instance, \cite{P} and the references therein.

\begin{lemma}\label{lemma_lv}
Let $w(s)$ be a positive, continuous function on  $[0,L]$, even with respect to $\dfrac L 2$, and let $p>1$. Then 
\begin{equation}\label{Lv}
 \mu_1(0,L;w) \ge \frac{\displaystyle \min_{s \in [0,L]} w(s)  }{\dint_0^{\frac{L}{2}}\left(\frac{L}{2}-s\right)^{p-1}w(s)\,ds}.
\end{equation}
\end{lemma}

\begin{proof}
It is easy to show that if $0<L'<L$, then
\[
\underset{\underset{\dint_0^L|u|^{p-2}u w\,ds=0}{u\in W^{1,p}(0,L)\setminus\{0\}}}{\inf}\frac{\dint_0^L|u'(s)|^p w(s)\,ds}{\dint_0^L|u(s)|^p w(s)\,ds}<\underset{\underset{\dint_0^{L'}|u|^{p-2}u w\,ds=0}{u\in W^{1,p}(0,L')\setminus\{0\}}}{\inf}\frac{\dint_0^{L'}|u'(s)|^p w(s)\,ds}{\dint_0^{L'}|u(s)|^p w(s)\,ds}.
\]
So, if $u$ denotes an eigenfunction for $\mu_1(0,L;w)$, it follows that $u$ changes sign exactly once on the interval $(0,L)$; say $u>0$ on $(0,L_u)$ and $u<0$ on $(L_u,L)$. Defining
\[
F(s)=-w(s) |u'(s)|^{p-2} u'(s)=\mu_1(0,L;w)\int_0^s  |u(t)|^{p-2} u(t) w(t)\,dt,\qquad 0\leq s \leq L,
\]
note that $F(0)=F(L)=0$. Moreover, $F'>0$ on $(0,L_u)$ and $F'<0$ on $(L_u,L)$. It follows that $F>0$ on $(0,L)$ and so $u$ is strictly decreasing on $(0,L)$. By standard arguments  $\mu_1(0,L;w)$ is simple, and since $w$ is even with respect to $s=L/2$, the eigenfunction $u$ is odd with respect to $s=L/2$. Therefore
\begin{equation}\label{eq:mustarhalf}
\mu_1(0,L; \delta)=\frac{\dint_0^L|u'(s)|^p w(s)\,ds}{\dint_0^L|u(s)|^p w(s)\,ds}=\frac{\dint_0^{\frac{L}{2}}|u'(s)|^p w(s)\,ds}{\dint_0^{\frac{L}{2}}|u(s)|^p w(s)\,ds}.
\end{equation}
By H\"older's inequality, we see that for $0\leq t\leq \frac{L}{2}$,
$$
|u(t)|\leq \int_t^{L/2}|u'(s)|\,ds\\
\leq \left(\frac{L}{2}-t\right)^{\frac{1}{p'}}\left(\int_0^{\frac{L}{2}}|u'(s)|^p\,ds\right)^\frac{1}{p},
$$
where $p'$ is the conjugate exponent to $p$. Raising both sides to power $p$ and multiplying by $\displaystyle \left( \min_{t \in [0,L]}w(t) \right)w(t)$, we get
\[
\displaystyle \left( \min_{t \in [0,L]}w(t) \right) |u(t)|^pw(t) \leq \left(\frac{L}{2}-t\right)^{p-1}w(t)\left(\int_0^{\frac{L}{2}}|u'(s)|^p w(s)\,ds\right).
\]
Integrating this inequality from $t=0$ to $t=\frac{L}{2}$ and using \eqref{eq:mustarhalf} completes the proof.
\end{proof}

\begin{remark}
We observe that on one hand \eqref{Lv} does not require any concavity assumption on $w$, but on the other hand it is a bit weaker than \eqref{fnt}. Indeed, recalling that $w$ is even with respect to $\frac L 2$, we have  that for every $p>1$
\[
\frac{ \displaystyle \min_{s \in [0,L]} w(s) }{\dint_0^{\frac{L}{2}}\left(\frac{L}{2}-s\right)^{p-1}w(s)\,ds} 
\le
p\left(\frac 2 L\right)^p,
\]
with equality when $w(s)$ is constant. We can numerically verify that
$$
2p^{\frac 1 p}<\pi_p \quad \Longleftrightarrow \quad \frac{\sin\left(\frac \pi p\right)}{\frac \pi p}<\left(1-\frac 1 p\right)^{\frac 1 p}, \quad p>1.
$$
We plot both sides of the above inequality in Figure 2, where  $x=\frac 1 p$.

\begin{figure}[h]
\begin{centering}
\includegraphics[scale=0.43]{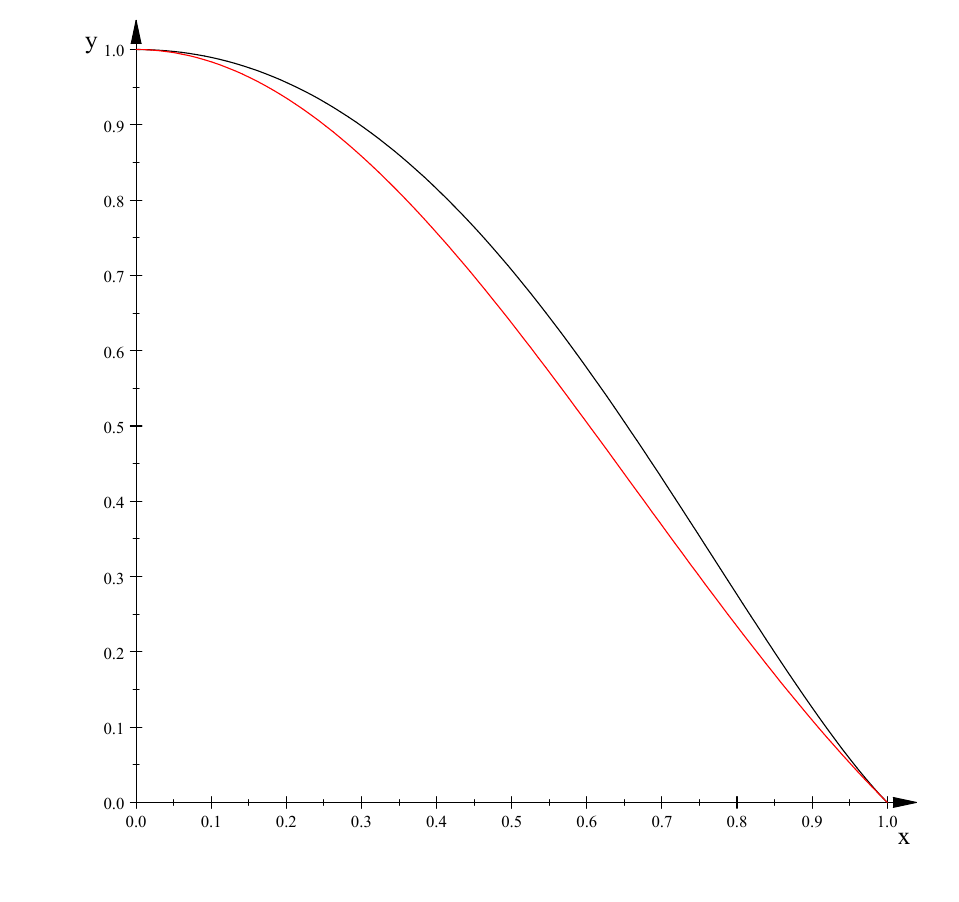}
\includegraphics[scale=0.43]{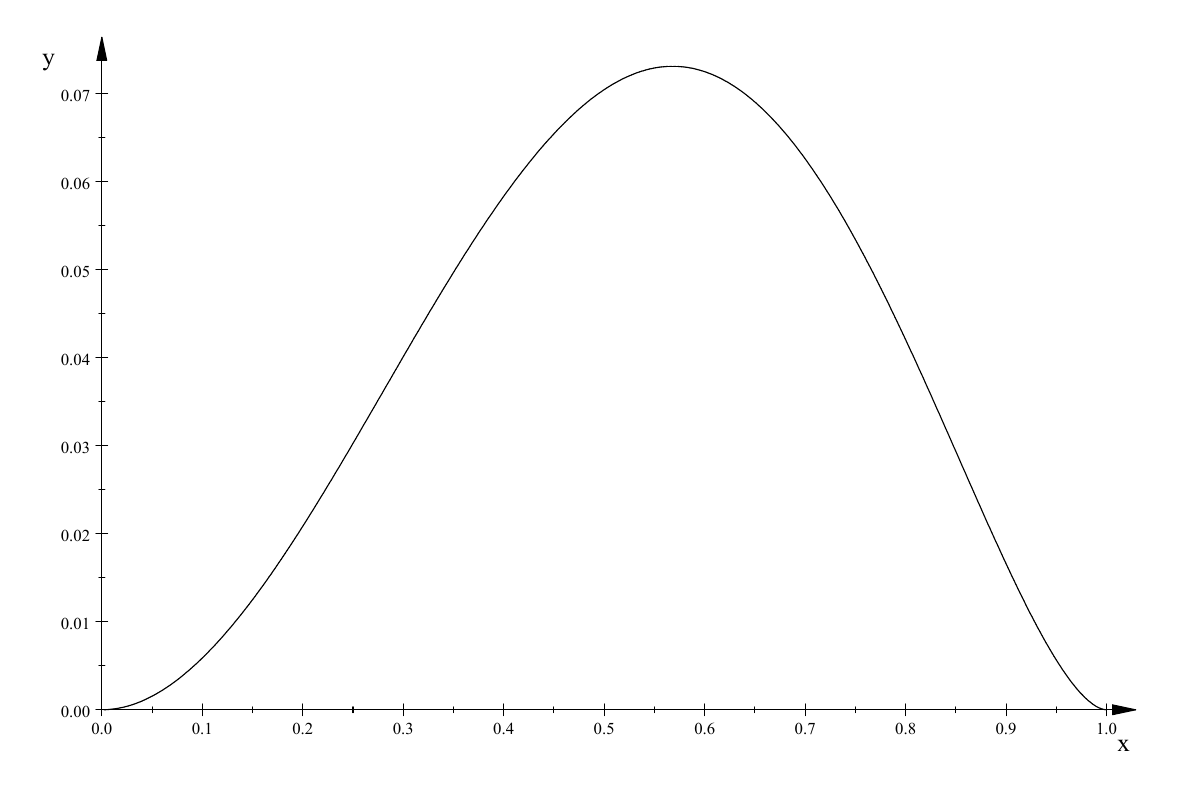}
\end{centering}
\caption{Left: in red the graph of $r(x)=\frac{\sin (\pi x)}{\pi x}$, in black the graph of $b(x)=(1-x)^x$, $x \in [0,1]$.  Right: the graph of $b(x)-r(x)$, $x \in [0,1]$.}
\end{figure}
\end{remark}

\smallskip

We also recall here two easy inequalities that will be used in Section 4.
\begin{proposition}
Let $p>1$. For every $a,b \geq 0$ the following inequalities hold true:
\begin{equation*}
\left(a^2+b^2\right)^{p/2}\ge C_p\left(a^p+b^p\right),
\end{equation*}
with
\begin{equation}
C_p=\left\{
\begin{array}{ll}
2^{\frac{p-2}{2}} \quad & \mbox{if } 1<p<2,
\\
1 & \mbox{if } p \ge 2,
\end{array}
\right. \label{cp}
\end{equation}
and
\begin{equation}
a^p+b^p \ge 2^{1-p}(a+b)^p. \label{cp1}
\end{equation}
\end{proposition}

\medskip

\section{Proof of Theorem \ref{t2}}

Our proof is inspired by that of Proposition 3.2 in \cite{BCDL}.

\begin{proof}[Proof of Theorem \ref{t2}]
Arguing by contradiction, suppose that there are no odd eigenfunctions for $\mu _{1}(D)$. Let $v(x,y)$ denote a fixed eigenfunction for $\mu _{1}(D)$ and define
\[
u(x,y)=v(x,y)+v(-x,y).
\]
Then $u$ is not identically zero, and hence is an even eigenfunction for $\mu_1(D)$. From now on we will denote by $D^+=\{u>0\}$, by $\Gamma=\{u=0\}$ the nodal set of $u$, by $\gamma^\delta$ the curve
$$
\gamma^\delta=\{(x(s)+\delta(s)y'(s),y(s)-\delta(s)x'(s)):s \in [0,L]\}.
$$ 
Observe that $\partial D =\gamma\cup\gamma^\delta\cup \mathcal{S}$, where $\mathcal{S}$ denotes the union of two symmetric line segments connecting the curves $\gamma$ and $\gamma^\delta$. 
We first recall that $\Gamma$ cannot be a closed curve, since otherwise, if we denote by $\lambda_1(D^+)$ the first eigenvalue  of the Dirichlet-Laplacian in $D^+$ we get
$$
\mu_1(D)=\lambda_1(D^+)>\lambda_1(D)>\mu_1(D),
$$
reaching a contradiction. By Courant's Theorem  on the number of nodal domains, we immediately exclude the possibility that $\Gamma$ could join $\gamma$ and $\gamma^\delta$. Precisely one of the following possibilities therefore occurs (see Figure 3):
\begin{itemize}
\item[$(i)$] $\partial D \cap \Gamma=\{P_L,P_R\} \subset \gamma$;
\item[$(ii)$] $\partial D \cap \Gamma=\{P_L,P_R\} \subset \mathcal{S}$; or
\item[$(iii)$] $\partial D \cap \Gamma=\{P_L,P_R\} \subset  \gamma^\delta$. 
\end{itemize}

\begin{figure}[h]
\includegraphics[scale=0.229]{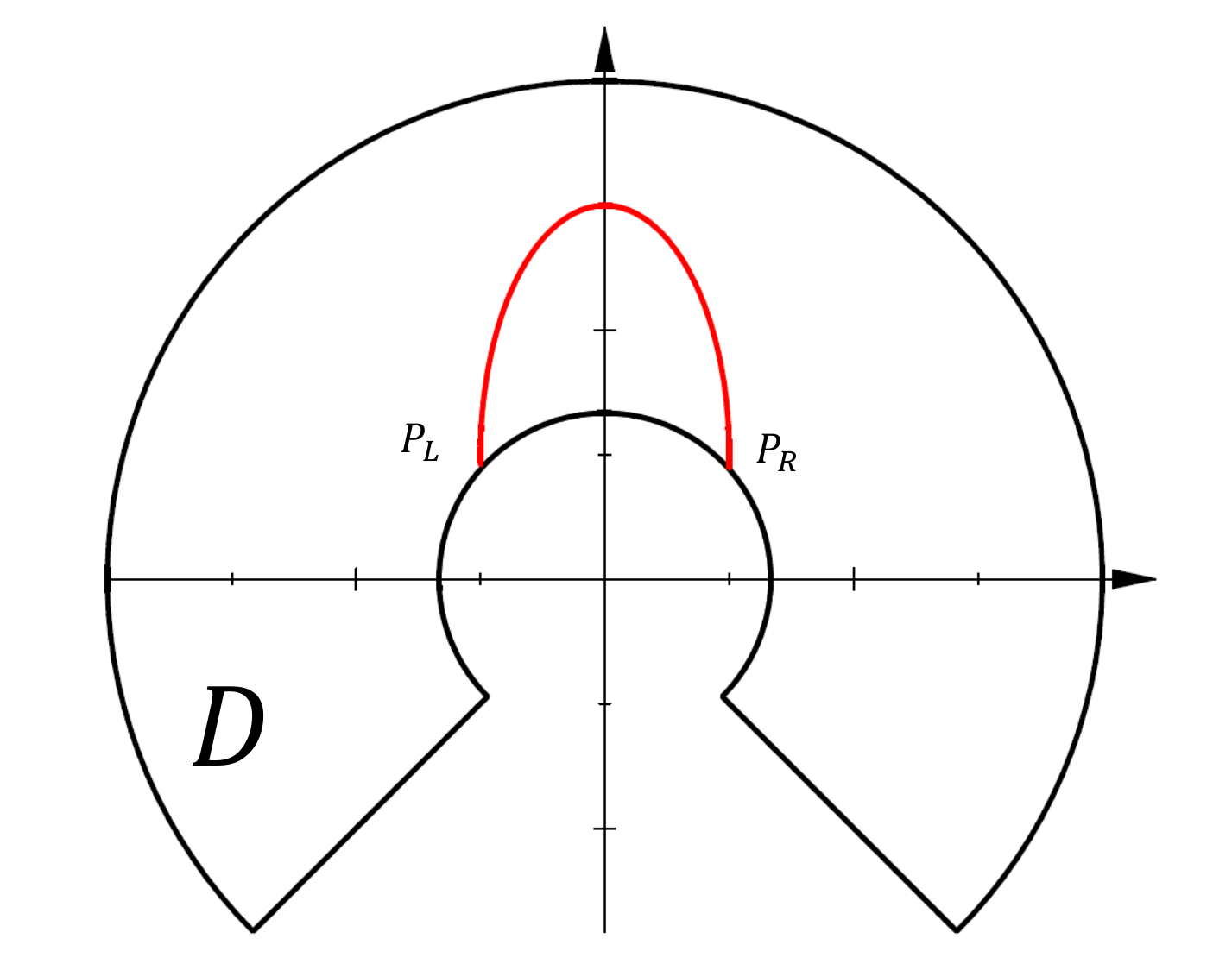}
\includegraphics[scale=0.229]{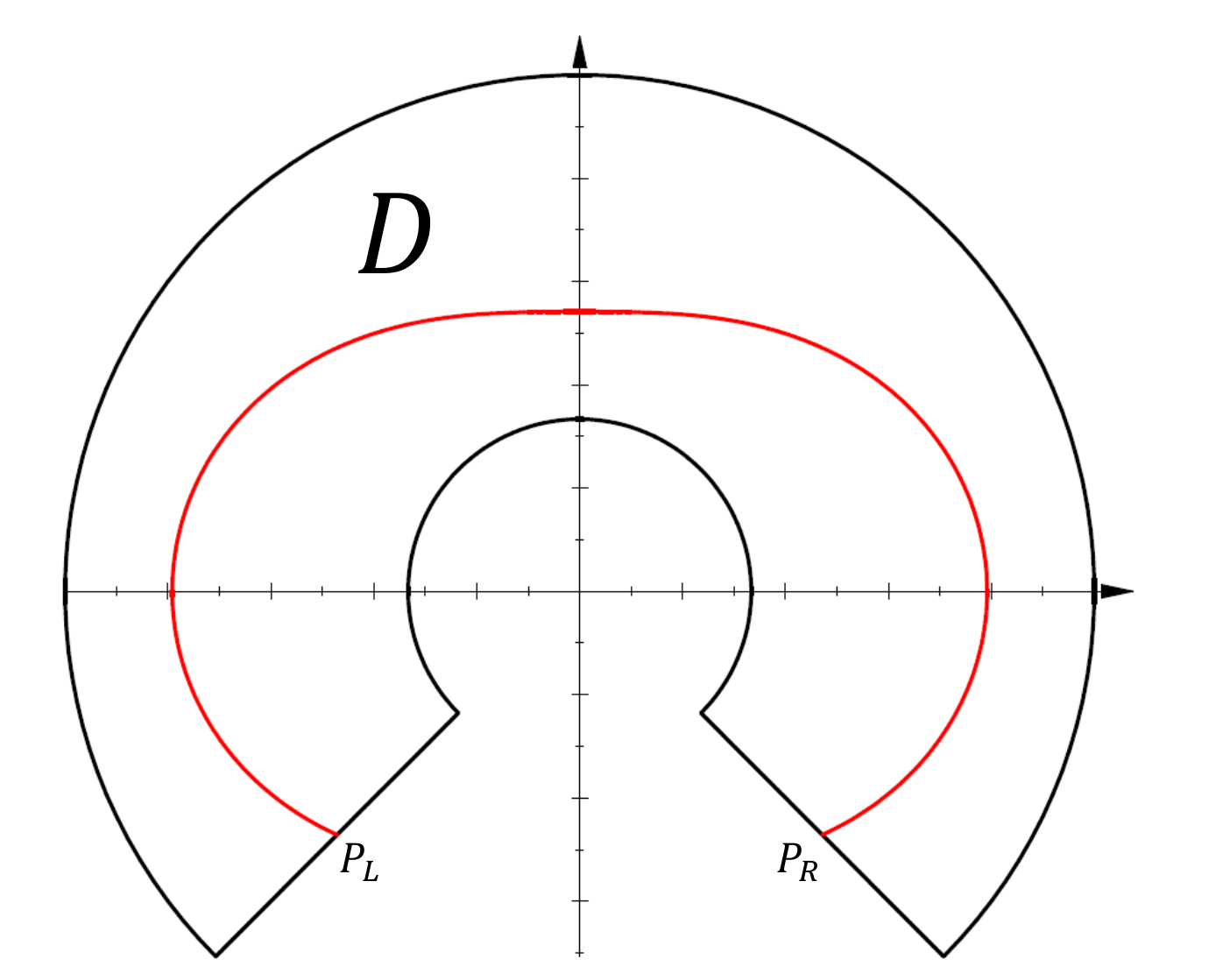}
\includegraphics[scale=0.229]{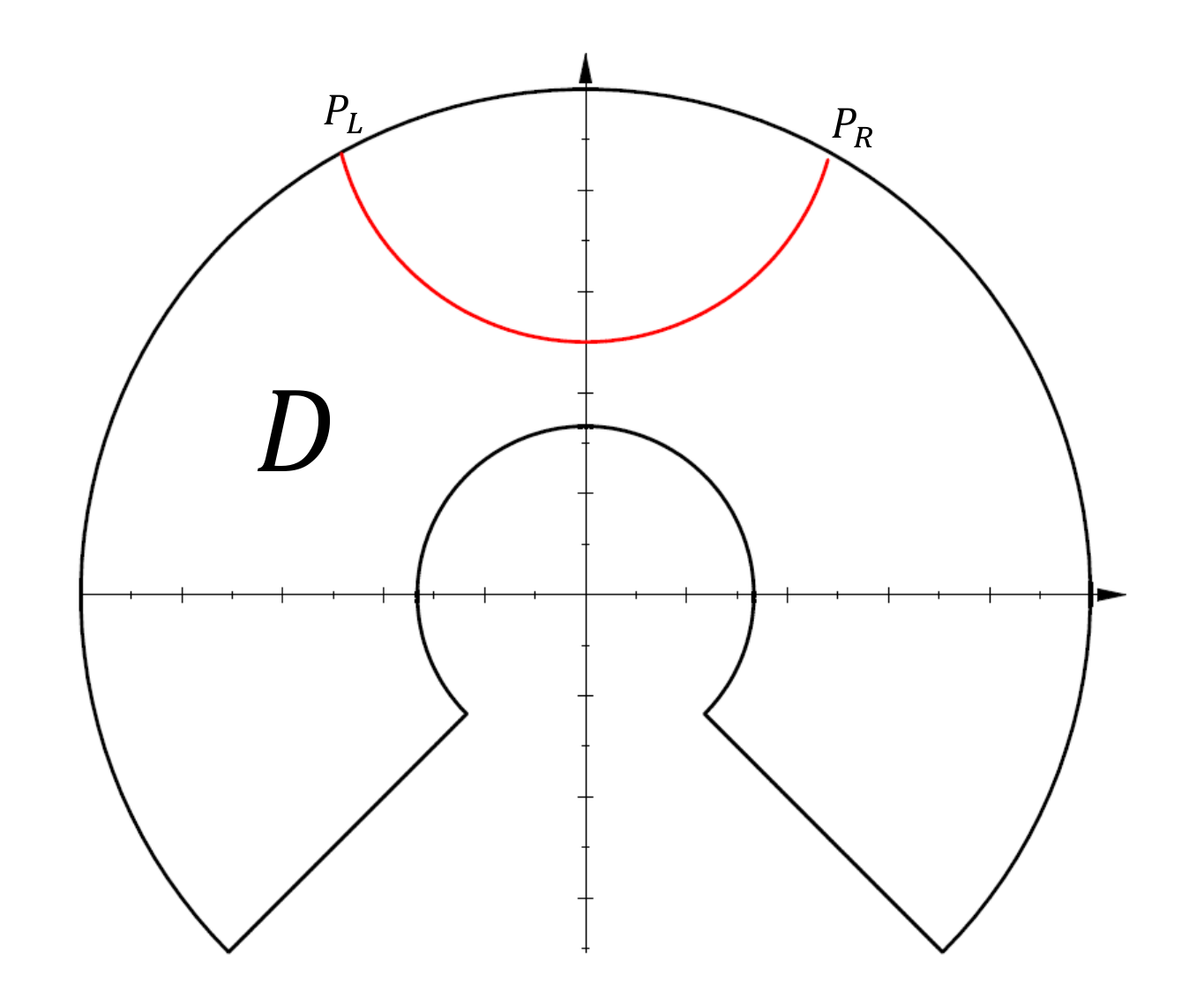}
\caption{On the left: Case $(i)$; in the middle: Case $(ii)$; on the right: Case $(iii)$}
\end{figure}

Clearly, due to the symmetry of $u$, $P_L$ and $P_R$ are symmetric to each other with respect to the $y$-axis. We distinguish three possibilities: $(a)$ $k(s) \ge 0$ for all $s \in [0,L]$; $(b)$ $k(s)<0$ for all $s \in [0,L]$; $(c)$ $k$ changes its sign in $[0,L]$.

$(a,i)$ We consider the mixed Dirichlet-Neumann eigenvalue problem
\begin{equation}\label{ND}
\left\{
\begin{array}{ll}
-\Delta \psi = \lambda \psi & \mbox{in}\> D,
\\ \\
\frac{\partial \psi}{\partial \mathbf{n}}=0 &\mbox{on} \>\wideparen{P_LP_R},
\\ \\
\psi=0 & \mbox{on} \>\partial D\setminus \wideparen{P_LP_R}.
\end{array}
\right.
\end{equation}
Here, $\wideparen{P_LP_R}$ denotes the portion of $\gamma$ connecting $P_L$ and $P_R$. By replacing $u$ with $-u$ if necessary, we assume $\partial D_+\cap \gamma =\wideparen{P_LP_R}$.  Write $\lambda^{ND}(D)$ for the lowest eigenvalue of problem \eqref{ND} with $\psi$ a corresponding eigenfunction. Then $u_+$ (the positive part of $u$) is a  valid trial function in the Rayleigh quotient for $\lambda^{ND}(D)$, so we see
\begin{equation*}\label{bb}
\mu_1(D)=\dfrac{\dint_{D_+} |\nabla u|^2 dxdy}{\dint_{D_+} u^2 dxdy}=\dfrac{\dint_{D} |\nabla  u_+|^2 dxdy}{\dint_{D}  u_+^2 dxdy} \geq \lambda^{ND}(D)=\dfrac{\dint_{D} |\nabla \psi|^2 dxdy}{\dint_{D} \psi^2 dxdy}.
\end{equation*}
Passing to Fermi coordinates, the estimate above yields
\begin{equation}\label{Fubini1}
\mu _{1}(D)\geq \frac{\dint_0^L \left(\dint_0^{\delta(s)} \psi_r^2 (1+rk(s))dr\right)ds}{\dint_0^L \left(\dint_0^{\delta(s)} \psi^2 (1+rk(s))dr\right)ds}.
\end{equation}
Observe that $\psi(\delta(s),s)=0$ for $s\in [0,L]$. Fix $s \in [0,L]$ and extend $\psi$ and $1+rk(s)$ from $[0,\delta(s)]$ to $[0,2\delta(s)]$ by odd and even reflection across $r=\delta(s)$, respectively. The even (linear) extension of $1+rk(s)$ is concave in $r$, courtesy of the assumption $k\geq 0$. An application of Lemma \ref{pw} gives 
$$
\int_0^{\delta(s)}\psi_r^2(1+rk(s))dr\geq \frac{\pi^2}{4\delta(s)^2}\int_0^{\delta(s)}\psi^2(1+rk(s))dr
$$
which implies
\begin{equation*}
\dint_0^L\left(\dint_0^{\delta(s)} \psi_r^2 (1+rk(s))dr\right)ds \ge  \dfrac{\pi^2}{4\underset{s\in [0,L]}{\max}\delta(s)^2}\dint_0^L\left(\dint_0^{\delta(s)} \psi^2 (1+rk(s))dr\right)ds.
\end{equation*}
Taken in sum, the inequality immediately above, inequality \eqref{Fubini1}, and $k\geq 0$ together imply
\begin{equation*}\label{lb1}
\mu_1(D) \ge \frac{\pi^2}{4\underset{s\in [0,L]}{\max}\delta(s)^2} \ge \dfrac{1}{\underset{s\in [0,L]}{\max} (2\delta(s)+\delta(s)^2 k(s))^2},
\end{equation*}
contradicting \eqref{delta}.

$(a, ii-iii)$ In case $(ii)$ let $\wideparen{P_LP_R}$ denote the portion of $\partial D$ connecting the points $P_L$ and $P_R$ which intersects $\gamma^\delta$; in case $(iii)$ let $\wideparen{P_LP_R}$ denote the portion of $\gamma^\delta$ connecting $P_L$ to $P_R$. We continue to let $\lambda^{ND}(D)$ denote the lowest eigenvalue of problem \eqref{ND} with our new definition of $\wideparen{P_LP_R}$. Arguing as in case $(a,i)$, we assume that $\partial D_+\cap (\gamma^\delta\cup\mathcal{S})=\wideparen{P_LP_R}$. Moreover, the arguments used to establish \eqref{Fubini1} carry over without change. Since $\wideparen{P_LP_R}\cap \gamma=\emptyset$, $\psi$ vanishes whenever $r=0$.  Fix $s \in [0,L]$. If $k(s)=0$, extend $\psi$ to $[-\delta(s),\delta(s)]$ by odd reflection through $r=0$. Lemma \ref{pw} gives 
\begin{equation}\label{iii}
\int_0^{\delta(s)}\psi_r^2dr\geq \frac{\pi^2}{4\delta(s)^2}\int_0^{\delta(s)}\psi^2dr.
\end{equation}
If $k(s)>0$, we must work a bit harder. To this end, define
\begin{eqnarray}\label{b1squared}
B_1(s)^2&=&\max_{r \in [0,\delta(s)] }\left(\int_r^{\delta(s)} (1+tk(s))dt\right)\left(\int_0^r \dfrac{1}{1+tk(s)}dt\right)
\\
&=&
\displaystyle\max_{r \in [0,\delta(s)] } (\delta(s) -r)\left(1+(\delta(s)+r)\dfrac{k(s)}{2}\right)\dfrac{\log(1+rk(s))}{k(s)}. \nonumber
\end{eqnarray}
The estimate provided by \cite[p. 40, Thm. 1]{M} gives
\begin{equation}\label{bar}
\dint_0^{\delta(s)} \psi_r^2 (1+rk(s))dr \ge C_1(s) \dint_0^{\delta(s)} \psi^2 (1+rk(s))dr ,
\end{equation}
where 
$$ C_1(s) \ge \dfrac{1}{4B_1(s)^2}.$$
A computation shows that
\begin{equation}\label{b11}
 B_1(s)^2\le   \left(\delta(s)+\delta(s)^2\dfrac{k(s)}{2}\right)^2
\end{equation}
from which we deduce
\begin{equation}\label{cs}
C_1(s) \ge \dfrac{1}{(2\delta(s)+\delta(s)^2 k(s))^2}.
\end{equation}

We can summarize \eqref{iii}, \eqref{bar}, and \eqref{cs} as follows:
$$
\dint_0^{\delta(s)} \psi_r^2 (1+r k(s))dr \ge  \left\{\begin{array}{ll} \dfrac{\pi^2}{ 4\delta(s)^2}\dint_0^{\delta(s)} \psi^2 (1+rk(s))dr & \mbox{if}\> k(s)=0, \\ \\ \dfrac{1}{(2\delta(s)+ \delta(s)^2 k(s))^2}\dint_0^{\delta(s)} \psi^2 (1+ r k(s))dr & \mbox{if}\> k(s)>0. \end{array}\right. 
$$
Hence, for $k(s)\geq0$ we have
$$
\dint_0^{\delta(s)} \psi_r^2 (1+r k(s))dr \geq \frac{1}{(2\delta(s)+\delta(s)^2 k(s))^2} \dint_0^{\delta(s)} \psi^2 (1+rk(s))dr.
$$
The inequality immediately above and inequality \eqref{Fubini1} together imply
\begin{equation*}\label{lb2}
\mu_1(D) \geq \dfrac{1}{\underset{s\in [0,L]}{\max} (2\delta(s)+\delta(s)^2 k(s))^2}
\end{equation*}
which again contradicts \eqref{delta}.

\medskip

$(b,i)$   The work of case $(a,i)$ applies verbatim to give \eqref{Fubini1}. We again employ the estimate \cite[p. 40, Thm. 1]{M} to bound the right-hand side of \eqref{Fubini1} .  We fix $s \in [0,L]$ and choose $\mu=\nu=(1 + (\delta(s)-x)k(s))\textbf{1}_{[0,\delta(s)]}(x)$ in this result.  If we define
\begin{eqnarray*}
B_2(s)^2 &=&\max_{r \in [0,\delta(s)] }\left(\int_0^{\delta(s)-r} (1+t k(s))dt\right)\left(\int_{\delta(s)-r}^{\delta(s)} \dfrac{1}{1+t k(s)}dt\right)
\\
&=&\max_{r \in [0,\delta(s)] } (\delta(s) -r)\left(1+(\delta(s)-r)\dfrac{k(s)}{2}\right)\dfrac{1}{(-k(s))}\log\left(\dfrac{1+(\delta(s)-r)k(s)}{1+\delta(s) k(s)}\right),
\end{eqnarray*}
then
\begin{equation}\label{bar2}
\dint_0^{\delta(s)} \psi_r^2 (1+r k(s))dr \ge C_2(s) \dint_0^{\delta(s)} \psi^2 (1+r k(s))dr,
\end{equation}
with
$$ C_2(s) \ge \dfrac{1}{4B_2(s)^2}.$$
It can be easily shown that
\[
B_2(s)^2\leq \frac{\delta(s)^2}{(1+\delta(s)k(s))^2}
\]
and so
$$
C_2(s) \ge \dfrac{(1+\delta(s) k(s))^2}{4\delta(s)^2}.
$$
The inequality immediately above and \eqref{bar2} together imply
$$
\dint_0^{\delta(s)} \psi_r^2 (1+rk(s))dr \ge \dfrac{(1+\delta(s) k(s))^2}{ 4\delta(s)^2} \dint_0^{\delta(s)} \psi^2 (1+ r k(s))dr.
$$
Pairing this inequality with \eqref{Fubini1} gives
\begin{equation*}\label{lb3}
\mu_1(D) \ge \min_{s\in [0,L]}\dfrac{(1+\delta(s) k(s))^2}{4\delta(s)^2},
\end{equation*}
which contradicts \eqref{delta1}.

$(b,ii-iii)$ Again, $\psi$ vanishes when $r=0$. Using the definition of $B_1(s)^2$ in equation \eqref{b1squared}, the estimate given in \eqref{b11} is replaced by the estimate
$$
B_1(s)^2 \le \dfrac{\delta(s)^2}{(1+\delta(s) k(s))^2},
$$
which again contradicts \eqref{delta1}.

\medskip
$(c)$ If the curvature $k$ changes its sign in $[0,L]$, we combine the work above from cases $(a,i), (a,ii-iii), (b,i), (b,ii-iii)$. 
\end{proof}

 \begin{remark}\label{remarkopen}
We stress that  \eqref{delta}, \eqref{delta1}, \eqref{delta2} give concrete conditions ensuring that $\mu_1(D)$ and $\mu_1^{odd}(D)$ coincide. All the arguments above work if one replaces $\mu_1(D)$ in  \eqref{Fubini1} with an upper bound, easily found by choosing a suitable test function. For instance, one  can choose $\cos\left(\dfrac{\pi}{L}s\right)$ as a test function in the following variational formulation of $\mu_1(D)$ written in Fermi coordinates:
$$
\mu_1(D)
=\min\left\{  \frac{\dint_0^L \left(\dint_0^{\delta(s)} \left[\dfrac{\psi_s^2}{(1+rk(s))^2}+\psi_r^2\right] (1+rk(s))dr\right)ds}{\dint_0^L \left(\dint_0^{\delta(s)} \psi^2 (1+rk(s))dr\right)ds}\right\},
$$
where the minimum is taken over the set of functions $\psi\in W^{1,2}(D)\setminus\{0\}$ having zero mean value in $D$, that is,
$$
\int_0^L\left(\int_0^{\delta(s)}\psi (1+rk(s))\,dr\right)ds=0.
$$
Hence we obtain
$$
\mu_1(D) \le \frac{\pi^2}{L^2} \frac{\dint_0^L\left(\dint_0^{\delta(s)} \sin\left(\frac{\pi}{L} s\right)^2\frac{1}{1+rk(s)}dr\right)ds}{\dint_0^L\left(\dint_0^{\delta(s)} \cos\left(\frac{\pi}{L} s\right)^2(1+rk(s))dr\right)ds}.
$$
Therefore, to ensure $\mu_1(D)=\mu_1^{odd}(D)$, we can replace inequality \eqref{delta} with 
 \begin{equation*}\label{delta11}
\underset{s\in[0,L]}{\max} (2\delta(s)+\delta(s)^2 k(s))^2
<  \dfrac{L^2}{\pi^2}\> \frac{\dint_0^L\left(\dint_0^{\delta(s)} \cos\left(\frac{\pi}{L} s\right)^2(1+rk(s))dr\right)ds}{\dint_0^L\left(\dint_0^{\delta(s)} \sin\left(\frac{\pi}{L} s\right)^2\frac{1}{1+rk(s)}dr\right)ds}
\end{equation*}
in case (a); inequality \eqref{delta1} with
\begin{equation*}\label{delta12}
\underset{s\in[0,L]}{\max}\dfrac{4\delta(s)^2}{ (1+\delta(s) k(s))^2}<  \dfrac{L^2}{\pi^2}\> \frac{\dint_0^L\left(\dint_0^{\delta(s)} \cos\left(\frac{\pi}{L} s\right)^2(1+rk(s))dr\right)ds}{\dint_0^L\left(\dint_0^{\delta(s)} \sin\left(\frac{\pi}{L} s\right)^2 \frac{1}{1+rk(s)}dr\right)ds}; 
\end{equation*}
in case (b); inequality \eqref{delta2} with
\begin{eqnarray}\label{delta13}
&\max\left\{  \underset{s\in[0,L]}{\max} (2\delta(s)+\delta(s)^2 k(s))^2,  \underset{s\in[0,L]}{\max}\dfrac{4\delta(s)^2}{ (1+\delta(s) k(s))^2}\right\}
\\
&\hskip6cm<  \dfrac{L^2}{\pi^2}\> \frac{\dint_0^L\left(\dint_0^{\delta(s)} \cos\left(\frac{\pi}{L} s\right)^2(1+rk(s))dr\right)ds}{\dint_0^L\left(\dint_0^{\delta(s)} \sin\left(\frac{\pi}{L} s\right)^2\frac{1}{1+rk(s)}dr\right)ds} \notag
\end{eqnarray}
in case (c).
 \end{remark}

\medskip

\section{Proofs of Theorems \ref{t11} and \ref{t1}}\label{sect:t1}

Our strategy is to first prove Theorem \ref{t1}. We then use some of the ideas presented there to establish Theorem \ref{t11}.

\begin{proof}[Proof of Theorem \ref{t1}]

Using the density of smooth functions up to the boundary in Sobolev spaces on Lipschitz domains, it will be enough to prove that
$$
\dfrac{\dint_D |\nabla u|^p \,dxdy}{\dint_D |u|^p\, dxdy} \ge B_p\, \left(\frac{\pi_p}{L}\right)^p
$$
when $u$ is odd and smooth on $\overline{D}$. Let $u$ be any such function. We write $D$ in Fermi coordinates and we slice $D$ into thin pieces so that $u$ has zero mean on each slice.  

Fix $n \in \N$ and for any $i=0,\dots,n-1$, write
$$D_i^F=\left\{(s,r)\in \R^2: \> 0 < s < L, \> \frac{i\,\delta(s)}{n}< r < \frac{(i+1)\,\delta(s)}{n}\right\}.$$
Write $D_i$ for the corresponding subset of $D$ via the Fermi coordinate transformation. Since $u$ is odd, we have 
$$ 
\int_{D_i} |u|^{p-2}u\,dxdy=0 \quad \text{for all } i=0,\dots,n-1.
$$

\noindent Using Fermi coordinates, we estimate the energy of $u$ in any $D_i$ (see the Appendix to \cite{BCDL}):
\begin{eqnarray*}
\int_{D_i} |\nabla u|^p dxdy &=& \int_0^L\left(\int_{\frac{i\,\delta(s)}{n}}^{\frac{(i+1)\,\delta(s)}{n}}\left(\frac{1}{(1+rk(s))^2}u_s(s,r)^2+u_r(s,r)^2\right)^{p/2}(1+rk(s))\,dr\right)ds
\\
&\ge&\min\left\{1, \min_{\overline{D_i^F}}\frac{1}{(1+rk(s))^p}\right\}\int_0^L\left(\int_{\frac{i\,\delta(s)}{n}}^{\frac{(i+1)\,\delta(s)}{n}}\left(u_s(s,r)^2+u_r(s,r)^2\right)^{p/2}(1+rk(s))\,dr\right)ds
\\
&\ge&C_p\min\left\{1, \min_{\overline{D_i^F}}\frac{1}{(1+rk(s))^p}\right\}\int_0^L\left( \int_{\frac{i\,\delta(s)}{n}}^{\frac{(i+1)\,\delta(s)}{n}}\left(|u_s(s,r)|^p+|u_r(s,r)|^p\right)(1+rk(s))\,dr\right)ds,
\end{eqnarray*} 
with $C_p$ defined as in \eqref{cp}.
We denote
$$
\int_0^L\left( \int_{\frac{i\,\delta(s)}{n}}^{\frac{(i+1)\,\delta(s)}{n}}\left(|u_s(s,r)|^p+|u_r(s,r)|^p\right)(1+rk(s))\,dr\right)ds=I_{1,i}+I_{2,i}+I_{3,i}+I_{4,i}+I_{5,i},
$$
where
\begin{align*}
I_{1,i}&=\int_0^L\left(\int_{\frac{i\,\delta(s)}{n}}^{\frac{(i+1)\,\delta(s)}{n}}\left(|u_s(s,r)|^p-\left|u_s\left(s,\frac{i\,\delta(s)}{n}\right)\right|^p\right)(1+rk(s))dr\right)ds,\\
I_{2,i}&=\int_0^L\left(\int_{\frac{i\,\delta(s)}{n}}^{\frac{(i+1)\,\delta(s)}{n}}\left|u_s\left(s,\frac{i\,\delta(s)}{n}\right)\right|^p(1+rk(s))dr\right)ds =\int_0^L\left|u_s\left(s,\frac{i\,\delta(s)}{n}\right)\right|^p\frac{\delta(s)}{n}\left(1+\frac{1+2i}{2} \frac{\delta(s)}{n}k(s) \right)ds,\\
I_{3,i}&=\int_0^L\left(\int_{\frac{i\,\delta(s)}{n}}^{\frac{(i+1)\,\delta(s)}{n}}\left(|u_r(s,r)|^p-\left|u_r\left(s,\frac{i\,\delta(s)}{n}\right)\right|^p\right)(1+rk(s))dr\right)ds,\\
I_{4,i}&=\int_0^L\left(\int_{\frac{i\,\delta(s)}{n}}^{\frac{(i+1)\,\delta(s)}{n}}\left|u_r\left(s, \frac{i\,\delta(s)}{n}\right)\right|^p(1+rk(s))\,dr\right)ds-\int_0^L\left|u_r\left(s, \frac{i\,\delta(s)}{n}\right)\frac{i\,\delta'(s)}{n}\right|^p \frac{\delta(s)}{n}\left(1+\frac{1+2i}{2} \frac{\delta(s)}{n}k(s) \right)ds \\
&=\int_0^L\left|u_r\left(s, \frac{i\, \delta(s)}{n}\right)\right|^p\frac{\delta(s)}{n}\left(1+\frac{1+2i}{2} \frac{\delta(s)}{n}k(s) \right)\left(1-\frac{i^p\, |\delta'(s)|^p}{n^p}\right)\,ds,\\
I_{5,i}&=\int_0^L \left|u_r\left(s, \frac{i\, \delta(s)}{n}\right)\right|^p\frac{i^p\, |\delta'(s)|^p}{n^p}\frac{\delta(s)}{n}\left(1+\frac{1+2i}{2} \frac{\delta(s)}{n}k(s) \right)ds.
\end{align*}
Since we may assume $|u_s|^p$ and $|u_r|^p$ are Lipschitz continuous functions, then 
\begin{equation}\label{i1}
|I_{1,i}| \le O\left(\frac 1 n \right)\,|D_i|, \quad|I_{3,i}| \le O\left(\frac 1 n \right)\,|D_i|, \qquad i=0,\dots,n-1.
\end{equation}
Moreover, by assumption $|\delta'(s)|\le 1$ and $1+\frac{1+2i}{2}\frac{\delta(s)}{n}k(s)>0$ since we assumed $1 + rk(s) > 0$ on $\overline{D^F}$. Thus, we also get
\begin{equation}\label{new}
I_{4,i}\ge 0.
\end{equation}
\noindent Similarly, we denote 
$$
\int_{D_i} |u|^pdxdy=\int_0^L\left(\int_{\frac{i\,\delta(s)}{n}}^{\frac{(i+1)\,\delta(s)}{n}}|u(s,r)|^p(1+rk(s))\,dr\right)ds=J_{1,i}+J_{2,i},
$$
where
\begin{align*}
J_{1,i}&=\int_0^L\left(\int_{\frac{i\,\delta(s)}{n}}^{\frac{(i+1)\,\delta(s)}{n}}\left(|u(s,r)|^p-\left|u\left(s,\frac{i\,\delta(s)}{n}\right)\right|^p\right)(1+rk(s))dr\right)ds,\\
J_{2,i}&=\int_0^L\left(\int_{\frac{i\,\delta(s)}{n}}^{\frac{(i+1)\,\delta(s)}{n}}\left|u\left(s, \frac{i\,\delta(s)}{n}\right)\right|^p(1+rk(s))dr\right)ds\\
&=\int_0^L\left|u\left(s,\frac{i\,\delta(s)}{n}\right)\right|^p\frac{\delta(s)}{n}\left(1+\frac{1+2i}{2} \frac{\delta(s)}{n} k(s)\right)ds.
\end{align*}
Arguing as above, we have that 
\begin{equation}\label{J1est}
|J_{1,i}|\le O\left(\frac 1 n \right)\,|D_i|, \qquad i=0,\dots, n-1.
\end{equation}

\noindent We now use Lemma \ref{2} in order to compare $I_{2,i}+I_{5,i}$ with $J_{2,i}$. Since $u$ is odd with respect to $\frac{L}{2}$ and $k(s)$ and $\delta(s)$ are even with respect to $\frac{L}{2}$, it follows that
$$
\int_0^L\left|u\left(s,\frac{i\,\delta(s)}{n}\right)\right|^{p-2}u\left(s,\frac{i\,\delta(s)}{n}\right)\,\frac{\delta(s)}{n}\left(1+\frac{1+2i}{2} \frac{\delta(s)}{n} k(s) \right)ds=0.
$$
As noted earlier, $1+\frac{1+2i}{2} \frac{\delta(s)}{n} k(s) >0$.  Recalling \eqref{cp1}, our concavity assumptions on $\delta(s)$ and either $\delta(s)k(s)$ or $\delta(s)^2k(s)$ paired with Lemma \ref{2} give
\begin{eqnarray}\label{i2}
I_{2,i}+I_{5,i}&\ge& 2^{1-p}\int_0^L\left|u_r\left(s,\frac{i\, \delta(s)}{n}\right)\frac{i\,\delta'(s)}{n}+u_s\left(s,\frac{i\,\delta(s)}{n}\right)\right|^p\frac{\delta(s)}{n}\left(1+\frac{1+2i}{2} \frac{\delta(s)}{n} k(s)\right)ds
\\
& \ge& 2^{1-p}\left(\frac{\pi_p}{L}\right)^pJ_{2,i}, \qquad i=0,\dots, n-1. \notag
\end{eqnarray}

\noindent Hence,  \eqref{i1}, \eqref{new}, and \eqref{i2} yield 
\begin{eqnarray*}
\int_{D_i}|\nabla u|^pdxdy &\ge&C_p\min\left\{1, \min_{\overline{D_i^F}}\frac{1}{(1+rk(s))^p}\right\}\left(I_{1,i}+I_{2,i}+I_{3,i}+I_{5,i}\right)
\\
&\ge&C_p\min\left\{1, \min_{\overline{D_i^F}}\frac{1}{(1+rk(s))^p}\right\}\left[  2^{1-p}\left(\frac{\pi_p}{L}\right)^pJ_{2,i}+O\left(\frac 1 n \right) |D_i|\right].
\end{eqnarray*}
Recalling the definition of $J_{2,i}$ and $B_p$ in \eqref{Bp}, we see
\begin{equation*}
\int_{D_i}|\nabla u|^pdxdy \ge B_p \left(\frac{\pi_p}{L}\right)^p\left(\int_{D_i}\left|u\right|^p\, dxdy+O\left(\frac 1 n \right)|D_i|\right)+ O\left(\frac 1 n \right)  |D_i|.
\end{equation*}
The slice estimate above hence gives the global estimate
$$
\int_D |\nabla u|^p dxdy \ge B_p \left(\frac{\pi_p}{L}\right)^p \int_D |u|^p dxdy +O\left(\frac 1 n \right) |D|,
$$
and claim follows by taking the limit as $n$ goes to $+\infty$.

\end{proof}

We next prove Theorem \ref{t11}

\begin{proof}[Proof of Theorem \ref{t11}]
We let $u$ and $D_i^F$ (with $\delta(s)=\delta)$ be as in the proof of Theorem \ref{t1}. We then have
\begin{eqnarray*}
\int_{D_i} |\nabla u|^p dxdy &=& \int_0^L\left(\int_{\frac{i\,\delta}{n}}^{\frac{(i+1)\,\delta}{n}}\left(\frac{1}{(1+rk(s))^2}u_s(s,r)^2+u_r(s,r)^2\right)^{p/2}(1+rk(s))\,dr\right)ds
\\
&\ge& \min_{\overline{D_i^F}}\frac{1}{(1+rk(s))^p}\int_0^L\left(\int_{\frac{i\,\delta}{n}}^{\frac{(i+1)\,\delta}{n}}|u_s(s,r)|^p(1+rk(s))\,dr\right)ds
\\
&\ge&A_p\int_0^L |u_s(s,r)|^p \frac{\delta}{n}\left(1+\frac{1+2i}{2}\frac{\delta}{n} k(s)\right)\, ds,
\end{eqnarray*}
with $A_p$ defined in \eqref{A_p}. With $I_{1,i}, I_{2,i}, J_{1,i}, J_{2,i}$ as in the proof of Theorem \ref{t1}, Lemma \ref{2} gives
\begin{equation}\label{i21}
I_{2,i} \ge \left(\frac{\pi_p}{L}\right)^pJ_{2,i}, \qquad i=0,\dots, n-1.
\end{equation}
Combining the estimate for $I_{1,i}$ in \eqref{i1} with \eqref{i21} yields
\begin{eqnarray*}
\int_{D_i}|\nabla u|^pdxdy &\ge&A_p\left(I_{1,i}+I_{2,i}\right)
\\
&\ge&A_p\left[ \int_0^L\left|u_s\left(s, \frac{i\,\delta}{n}\right)\right|^p\frac{\delta}{n}\left(1+\frac{1+2i}{2} \frac{\delta}{n} k(s) \right)ds+O\left(\frac 1 n \right) |D_i|\right]
\\
&\ge& A_p\left[ \left(\frac{\pi_p}{L}\right)^p\int_0^L\left|u\left(s, \frac{i\,\delta}{n}\right)\right|^p \frac{\delta}{n} \left(1+\frac{1+2i}{2} \frac{\delta}{n} k(s) \right)ds+O\left(\frac 1 n \right) |D_i|\right].
\end{eqnarray*}
Recalling the definition of $J_{2,i}$ and the estimate provided by \eqref{J1est}, we see
\begin{equation*}
\int_{D_i}|\nabla u|^pdxdy \ge A_p \left(\frac{\pi_p}{L}\right)^p\left(\int_{D_i}\left|u\right|^p\, dxdy+O\left(\frac 1 n \right)|D_i|\right)+ O\left(\frac 1 n \right)  |D_i|.
\end{equation*}
As before, we obtain the global estimate
$$
\int_D |\nabla u|^p dxdy \ge A_p \left(\frac{\pi_p}{L}\right)^p \int_D |u|^p dxdy +O\left(\frac 1 n \right) |D|,
$$
and the claim follows by sending $n \to +\infty$.
\end{proof}

\begin{remark}
As observed in \cite{BCDL}, when $p=2$ the estimate \eqref{est11} is sharp since equality is achieved when $D$ is a rectangle.
\end{remark}


\medskip

\section{Proof of Theorem \ref{th:4}}

In this section we study the asymptotic behavior of $\mu_1(D)$ as the width function goes to 0. 
For related results when $p=2$ and the width is constant, see for instance \cite{G,K,KRRS,KT,NW} and the references therein.

\begin{proof}[Proof of Theorem \ref{th:4}]
Let $u(s)$ be an eigenfunction corresponding to the first nonzero eigenvalue $\mu_1(0,L;\delta)$ of problem \eqref{eq:SLJ}.
Since, as we observed in the proof of Lemma \ref{lemma_lv}, $u(s)$ is odd with respect to $\dfrac L 2$, it can be used as a  trial function in the Rayleigh quotient for $\mu_1(D_{\varepsilon})$:
\[
\mu_1(D_{\varepsilon})\leq \frac{\dint_0^L \dint_0^{\varepsilon \delta(s)}|u'(s)|^p\frac{1}{(1+rk(s))^{p-1}}\,dr\,ds}{\dint_0^L\dint_0^{\varepsilon \delta(s)}|u(s)|^p(1+rk(s))\,dr\,ds}.
\]
The Dominated Convergence Theorem gives
\begin{align}
\limsup_{\varepsilon \to 0^+} \mu_1(D_{\varepsilon}) &\leq \lim_{\varepsilon \to 0^+}\frac{\dint_0^L \int_0^{\varepsilon \delta(s)}|u'(s)|^p\frac{1}{(1+rk(s))^{p-1}}\,dr\,ds}{\dint_0^L\dint_0^{\varepsilon \delta(s)}|u(s)|^p(1+rk(s))\,dr\,ds} \label{eq:ub} \\
&=\lim_{\varepsilon \to 0^+}
\begin{cases}
\frac{\dint_0^L u'(s)^2\frac{\log(1+\varepsilon \delta(s)k(s))}{\varepsilon k(s)}\,ds}{\dint_0^Lu(s)^2\delta(s)\left(1+\frac{1}{2}\varepsilon \delta(s)k(s)\right)\,ds} & \textup{if }p=2,\\
\displaystyle \frac{1}{2-p}\frac{\dint_0^L |u'(s)|^p\frac{(1+\varepsilon \delta(s)k(s))^{2-p}-1}{\varepsilon k(s)}\,ds}{\dint_0^L|u(s)|^p\delta(s)\left(1+\frac{1}{2}\varepsilon \delta(s)k(s)\right)\,ds} & \textup{if }p \ne 2,
\end{cases} \nonumber
\\
&=\frac{\dint_0^L|u'(s)|^p\delta(s)\,ds}{\dint_0^L|u(s)|^p\delta(s) \,ds}\nonumber \\
&=\mu_1(0,L;\delta). \nonumber
\end{align}

Next let $u_{\varepsilon}$ denote an eigenfunction for $\mu_1(D_{\varepsilon})$ normalized so that
\begin{equation}\label{eq:unorm}
\int_0^L\int_0^{\varepsilon \delta(s)}|u_{\varepsilon}|^p(1+rk(s))\,dr\,ds=\varepsilon.
\end{equation}
Make the change of variable $t=\frac{r}{\varepsilon}$ and define
\[
v_{\varepsilon}(s,t)=u_{\varepsilon}(s,\varepsilon t),\quad (s,t)\in D^F.
\]
Using the normalization from \eqref{eq:unorm}, we see that
\begin{align}
\mu_1(D_{\varepsilon})&=\frac{\dint_0^L\dint_0^{\varepsilon \delta(s)}\left(\frac{1}{\left(1+rk(s)\right)^2}\left(u_{\varepsilon}\right)^2_s+\left(u_{\varepsilon}\right)^2_{r}\right)^{\frac{p}{2}}\left(1+rk(s)\right)\,dr\,ds}{\dint_0^L\dint_0^{\varepsilon \delta(s)}|u_{\varepsilon}|^p\left(1+rk(s)\right)\,dr\,ds} \nonumber \\
&=\int_0^L\int_0^{\delta(s)}\left(\frac{1}{\left(1+\varepsilon t k(s)\right)^2}\left(v_{\varepsilon}\right)^2_s+\frac{1}{\varepsilon^2}\left(v_{\varepsilon}\right)^2_{t}\right)^{\frac{p}{2}}\left(1+\varepsilon t k(s)\right)\,dt\,ds.\label{eq:mu1ve}
\end{align}
From \eqref{eq:ub}, it follows that the eigenvalues $\mu_1(D_{\varepsilon})$ are bounded as $\varepsilon \to 0^+$. Pairing this observation with our assumption that $1+rk(s)>0$ on $\overline{D^F}$, equations \eqref{eq:unorm} and \eqref{eq:mu1ve} imply that the functions $v_{\varepsilon}$ are bounded in $W^{1,p}(D^F)$. By the Banach-Alaoglu-Bourbaki and Rellich-Kondrachov Theorems, we may pass to a subsequence and assume the existence of a function $v_0\in W^{1,p}(D^F)$ where $v_{\varepsilon}\to v_0$ weakly in $W^{1,p}(D^F)$ and strongly in $L^p(D^F)$. We may also assume $v_{\varepsilon}\to v_0$ pointwise a.e. on $D^F$.

Note that from \eqref{eq:mu1ve}, it follows that
\begin{equation}\label{eq:bdve}
\int_0^L\int_0^{\delta(s)}|\left(v_{\varepsilon}\right)_t|^p\,dt\,ds\leq (\textup{const.})\varepsilon^p\mu_1(D_{\varepsilon}).
\end{equation}
Using weak convergence, we see
\begin{align}
\int_0^L\int_0^{\delta(s)}|\left(v_0\right)_t|^p\,dt\,ds &\leq \liminf_{\varepsilon \to 0^+} \int_0^L\int_0^{\delta(s)}|\left(v_{\varepsilon}\right)_t|^p\,dt\,ds, \label{eq:vtbd}\\
\int_0^L\int_0^{\delta(s)}|\left(v_0\right)_s|^p\,dt\,ds &\leq \liminf_{\varepsilon \to 0^+} \int_0^L\int_0^{\delta(s)}|\left(v_{\varepsilon}\right)_s|^p\,dt\,ds.\label{eq:vtbd2}
\end{align}
Taken in sum, inequalities \eqref{eq:bdve} and \eqref{eq:vtbd} give that $\left(v_0\right)_t=0$ a.e. and so $v_0=v_0(s)$ on $D^F$.

Returning to \eqref{eq:mu1ve}, notice that
\[
\mu_1(D_{\varepsilon}) \geq \int_0^L\int_0^{\delta(s)}|\left(v_{\varepsilon}\right)_s|^p\,dt\,ds+\int_0^L\int_0^{\delta(s)}\left(\frac{1}{(1+\varepsilon tk(s))^{p-1}}-1\right)|\left(v_{\varepsilon}\right)_s|^p\,dt\,ds.
\]
Since $\frac{1}{(1+\varepsilon tk(s))^{p-1}}\to 1$ uniformly on $D^F$, the inequality above and \eqref{eq:vtbd2} give
\begin{equation}\label{eq:mulowebd}
\liminf_{\varepsilon \to 0^+}\mu_1(D_{\varepsilon}) \geq \int_0^L\int_0^{\delta(s)}|v_0'(s)|^p\,dt\,ds=\int_0^L|v_0'(s)|^p\delta(s)\,ds.
\end{equation}
The normalization \eqref{eq:unorm} implies that
\[
\int_0^L\int_0^{\delta(s)}|v_{\varepsilon}|^p\left(1+\varepsilon t k(s)\right)\,dt\,ds=1
\]
and so $L^p$-convergence implies
\begin{equation}\label{eq:v0norm}
\int_0^L\int_0^{\delta(s)}|v_0(s)|^p\,dt \,ds=\int_0^L|v_0(s)|^p\delta(s)\,ds=1.
\end{equation}
Recall that the eigenfunctions $u_{\varepsilon}$ satisfy
\[
\int_0^L\int_0^{\varepsilon \delta(s)}|u_{\varepsilon}|^{p-2}u_{\varepsilon}(1+rk(s))\,dr\,ds=0,
\]
so our functions $v_{\varepsilon}$ satisy
\[
\int_0^L\int_0^{\delta(s)}|v_{\varepsilon}|^{p-2}v_{\varepsilon}(1+\varepsilon tk(s))\,dt\,ds=0.
\]
By using H\"older's inequality and  Vitali's Theorem we get
\begin{equation}\label{eq:v0mean}
\int_0^L\int_0^{\delta(s)}|v_0(s)|^{p-2}v_0(s)\,dt\,ds=\int_0^L|v_0(s)|^{p-2}v_0(s)\, \delta(s)\,ds=0.
\end{equation}

Since $\varepsilon \to 0^+$ was arbitrary, combining \eqref{eq:mulowebd}, \eqref{eq:v0norm}, and \eqref{eq:v0mean}, we see that
\begin{equation}\label{eq:lastbd1}
\liminf_{\varepsilon \to 0^+}\mu_1(D_{\varepsilon}) \geq \mu_1(0,L;\delta).
\end{equation}
Combining inequalities \eqref{eq:ub} and \eqref{eq:lastbd1} completes the proof.
\end{proof}

We summarize the main results of our paper in the following final remark.
\begin{remark}
With $\gamma$ and $\delta$ as in the Introduction and $1+rk(s)>0$ on $\overline{D^F}$,  if $D_{\varepsilon}$ is the domain with width function $\varepsilon \delta(s)$, then Lemma \ref{2}, Lemma \ref{lemma_lv}, and Theorem \ref{th:4} give
\[
\lim_{\varepsilon \to 0^+}\mu_1(D_{\varepsilon})=\mu_1(0,L;\delta)\ge 
\left\{
\begin{array}{ll}
\left(\dfrac{\pi_p}{ L} \right)^p \qquad &\mbox{ if $\delta(s)$ is log-concave in $(0,L)$,}
\\ \\
\dfrac{\displaystyle\min_{s \in [0,L]}\delta(s)}{\dint_0^{L/2} \left(\frac L 2 -s\right)^{p-1}\delta(s)\,ds} & \mbox{if $\delta(s)$ is continuous in $[0,L]$.}
\end{array}
\right.
\]
Moreover, if $p=2$, observe that
\[
\displaystyle \lim_{\varepsilon \to 0^+} \frac{1}{\displaystyle \max_{s\in [0,L]} \left(2\varepsilon \delta(s)+(\varepsilon \delta(s))^2k(s)\right)^2}=\displaystyle \lim_{\varepsilon \to 0^+} \frac{1}{\displaystyle \max_{s\in [0,L]} \frac{4(\varepsilon \delta(s))^2}{\left(1+\varepsilon \delta(s)k(s)\right)^2}}=+\infty.
\]
Hence, when paired together, Theorems \ref{t2} and \ref{th:4} give
\[
\mu_1^{odd}(D_{\varepsilon})=\mu_1(D_{\varepsilon})
\]
for all $\varepsilon$ sufficiently small. Thus, Theorems \ref{t11} and $\ref{t1}$ provide explicit lower bounds on $\mu_1(D_{\varepsilon})$ for all $\varepsilon$ sufficiently small, provided the additional assumptions of those results are satisfied.
\end{remark}


\section*{Acknowledgement}
The authors want to thank D. Krej\v{c}i\v{r}\'ik for useful discussions. The first author has been partially supported by FFR 2024 Barbara Brandolini, Universit\`a degli Studi di Palermo (Italy).
The second author has been partially supported by the PRIN project 2017JPCAPN (Italy) grant:
``Qualitative and quantitative aspects of nonlinear PDEs''.

\end{document}